\documentclass{sig-alternate}
\usepackage[utf8]{inputenc}
\usepackage{amsmath}
\usepackage{ifdraft}
\ifdraft{
    \usepackage{showlabels}
}{}
\usepackage{bm}
\usepackage{tikz,pgfplots}
\usepackage{flushend}
\newtheorem{thm}{Theorem}

\newcommand*{\Prob}[1]{\mathbb{P}(#1)}
\newcommand*{\intd}[0]{\mathrm{d}}

\begin{document}

\title{New Look at Finite Single Server Queue with Poisson Input and Semi-Markov Service Times}

\def\sharedaffiliation{%
    \end{tabular}
    \begin{tabular}{c}%
}

\numberofauthors{2}
 \author{
 \alignauthor
Krzysztof Rusek\\
       \email{krusek@agh.edu.pl}
 \alignauthor
Zdzis\l{}aw Papir\\
       \email{papir@kt.agh.edu.pl}
       \sharedaffiliation
       \affaddr{Department of Telecommunications}\\
       \affaddr{AGH University of Science and Technology}\\
       \affaddr{Krakow, Poland}
 }

\maketitle

\begin{abstract}
The mathematics of the finite single server queue with Poisson input and semi-Markov service times($M/SM/1/b$) is similar to that used for $BMAP/G/1/b$ systems.
This observation results in new analytical formulas for a queue size in the $M/SM/1/b$ system.
Both stationary and the transient solutions are considered.
\end{abstract}




\section{Introduction}
It was observed, see e.g~\cite{korelacja}, that packet lengths in IP networks have nonzero autocorrelation.
This correlation may affect the accuracy of the switching device models when it is not taken into account.
If the buffer is measured in bytes, then it is drained at rate $C$ B/s and the correlation is not important because the service time is equal to $1/C$.
On the other hand if the buffer is measured in packets, the service time depends on packet lengths. 
Therefore, the service times are correlated themselves.

The semi-Markov process, being a direct generalization of the Markov process~\cite{janssen2006applied}, is often used as a model of the correlated service process in a queue.
The queues with semi-Markov service times have bean studied for a long time~\cite{neutsSM}.
Most of the theoretical results were obtained for the infinite capacity system.
Recently the general finite system $BMAP/SM/1/b$ was explored using the imbedded Markov chain approach~\cite{springerlink:10.1023/A:1019831410410}.

In this paper we present a new approach to $M/SM/1/b$ system.
The mathematical framework from~\cite{ChydzinskiBMAP2006springer} allows us to to find closed formulas for the queue size in both the stationary and the transient states.
This is especially important since non-stationary characteristics are getting a growing attention(see~\cite{Chdzinski2007StochMod} for references).

The model presented in this paper is the first step in creating a general model of a buffer in a real device.
For this kind of application the more complex arrival process has to be introduced to the model.

This is a short paper presenting work in progress.
Section~\ref{sec:model} describes in details the modeled system.
The main result is presented in section~\ref{sec:result}.
The paper is concluded in section~\ref{sec:cons} where also the future work is discussed.

\section{Queue model}\label{sec:model}
In this paper we consider a single server FIFO queuing system with correlated service times, fed by the stationary Poisson process with rate $\lambda$.
The system capacity is finite and equal to~$b$ (including a customer currently being served).

Suppose that the server can be in one of $m$ different states $S=\{1,\ldots,m\}$ and in the $i$th state service times are i.i.d. random variables with CDF $F_i(t),  i=1,\ldots,m$.
After each service completion the state of the server is changed in such way that the state sequence forms a homogenous Markov chain.

Let $\{J_n\}, n\geq 0$ denotes the state after $n-1$ transitions and $G_n, n\geq 0 $ is the $(n-1)$th service time.
Supposing $G_0=0$, the described service process $(J_n,G_n)$ is a semi-Markov process~\cite{janssen2006applied} defined on a finite set of states $S=\{1,2\ldots,m\}$ by the matrix $P_{J_{n-1},j}(t)= \Prob{J_n=j,G_n \leq t|(J_{n-1},G_{n-1})}$.
Note that the semi-Markov process is constructed of service times only.
The idle periods are ignored, however we define $J(t) \in S$ being the server state at time $t$ including the idle periods.
We assume that the time origin corresponds to a departure epoch.

It is assumed that the state transition matrix of the underlaying Markov chain $\bm{T}=T_{i,j}=P_{i,j}(\infty)$ is known as a part of system parametrization.
The other parameters are the system capacity and the service time distribution in each state.
Besides the service time distributions ($F_i$), their Laplace transforms $f_i(s)=\int_0^\infty e^{-st}\mathrm{d}F_i(t), Re(s) >0
$ will also be used; in most cases in a convenient vector notation
$$
\bm f(s)=\left(\frac{1-f_1(s)}{s},\ldots,\frac{1-f_m(s)}{s}\right)^T.
$$

\section{Main Result}\label{sec:result}
Investigating the mathematics of some queue characteristics introduced by Chydzinski in \cite{Chdzinski2007StochMod, ChydzinskiBMAP2006springer} we observed some similarities
between $BMAP/G/1/b$ and $M/SM/1/b$ systems.
Since the potential method presented in  \cite{Chdzinski2007StochMod, ChydzinskiBMAP2006springer}is a powerful tool for finding some queue characteristics, e.g. queue occupancy distribution, we investigated the possibility of using it for the system with semi-Markov service times.
In this section we prove that indeed it is possible to use the potential method for a large subclass (those with nonsingular transition matrix) of such systems.

Let us start with the queue length at time $t$ denoted as $X(t)$.
Its transient probability distribution depends on both, the queue length and the server state at time $t=0$.
Therefore we define the conditional distribution
\begin{align*}
\Phi_{n,i}(t,l)=\Prob{X(t)=l| X(0)=n,J(0)=i}\\
 0\leq n \leq b,\quad 1\leq i \leq m.
\end{align*}
Let us also define the Laplace transform of $\Phi_{n,i}(t,l)$
\begin{equation*}
    \phi_{n,i}(s,l)=\int_{0}^{\infty} e^{-st}\Phi_{n,i}(t,l)\mathrm{d}t
\end{equation*}
and its convenient vector form
\begin{equation*}
    \bm{\phi}_{n}(s,l)=(\phi_{n,1}(s,l),\phi_{n,2}(s,l),\ldots,\phi_{n,m}(s,l))^T.
\end{equation*}

The obtained results are easier to present if we establish the following notation
\begin{align*}
    a_{i,k}(s)&=\int_{0}^{\infty}\frac{e^{-(\lambda + s)t}(\lambda t)^{k}}{k!} \, \mathrm{d}F_i(t),\\
    d_{i,k}(s)&=\int_{0}^{\infty}\frac{e^{-(\lambda + s)t}(\lambda t)^{k}}{k!}(1-F_i(t)) \, \mathrm{d}t,\\
    \bm d_k(s)&=(d_{1,k}(s),\ldots,d_{m,k}(s))^T
\end{align*}
\begin{align*}
    &\bm{A}_k(s) = [T_{i,j}a_{i,k}(s)]_{i,j}\\
    &\bm{\bar{A}}_k(s)=\sum _{i=k}^{\infty} \bm{A}_k(s);\\
	&\bm R_{0}=\bm{0} \\
    &\bm R_{1}=\bm A_{0}^{-1} \\
    &\bm R_{k+1}=\bm R_{1}(\bm R_{k}-\sum_{i=0}^{k}\bm A_{i+1}\bm R_{k-i}), \qquad k \geq1,\\
    &\bm B_n(s)=\bm A_{n+1}(s)-\bm{\bar{A}}_{n+1}(s)\bm{\bar{A}}_0^{-1}(s)
\end{align*}
\begin{align*}
    &\bm{M}_{b}=\bm R_{b+1}(s) \bm A_0(s) + \sum _{k=0}^{b}\bm R_{b-k}(s) \bm B_{k}(s)- \\
    &-\frac{\lambda }{s+\lambda }\left(\bm R_{b}(s) \bm A_0(s)+\sum _{k=0}^{b-1} \bm R_{b-1-k}(s) \bm B_k(s)\right)
\end{align*}

Note that in the presented notation some matrices, namely $\bm{A}_0(s)$ and $\bm{\bar{A}}_0(s)$,  are inverted.
It is easy to show that $\det{\bm{A}_0(s)}=\det{\bm T}\cdot\prod_{j=1}^m a_{j,0}(s)$ a well as  $\bm{\bar{A}}_0(s)=\det{\bm T}\cdot\prod_{j=1}^m f_j(s)$.
Therefore the assumption that $a_{j,0}(s)$ and $f_j(s)$ are nonzero (we also assume that $\bm{M}_{b}$ in nonsingular) implicates the fact that the matrix $\bm T$ is the only source of singularities.
Under this assumption we can prove the following theorem.

\begin{thm}\label{th:1}
    If the transition matrix  $\bm{T}$ is not singular then the Laplace transform of the queue length distribution in the $M/SM/1/b$ system has the form:
    \begin{align}\label{eq:th1}
        &\bm{\phi}_{n}(s,l)=\\\nonumber
        &=\left(\bm{R}_{b-n+1}(s)\bm{A}_0(s)+ \sum _{k=0}^{b-n} \bm{R}_{b-n-k}(s)\bm{B}_k(s)\right)\bm{M}_{b}^{-1}\bm{l}_b(s,l)\\ \nonumber
        &+\sum _{k=0}^{b-n} \bm{R}_{b-n-k}(s)\bm{g}_k(s,l),
    \end{align}

    where
    \begin{equation*}
        \bm g_k(s,l)=\bm{\bar{A}}_{k+1}(s)\bm{\bar{A}}_0^{-1}(s)\bm r_b(s,l)-\bm r_{b-k}(s,l)
    \end{equation*}
    \begin{equation*}
        \bm r_k(s,l)=
        \begin{cases}
            \bm 0, \quad l < k \\
            \bm d_{l-k}(s) \quad k \leq l < b \\
            \bm f(s)-\sum_{i=0}^{b-k-1}\bm d_{i}(s), \quad l=b
        \end{cases}
    \end{equation*}
\end{thm}
\begin{proof}
    If the system is not empty at the beginning then from the total probability theorem applied to the first service completion time we get
    \begin{align}\label{proof11}
        &\Phi_{n,i}(t) =\\\nonumber
        &\sum_{\substack{k=0\\j=1}}^{\substack{b-n-1\\m}}\int_0^t \Phi_{n+k-1,j}(t-u)T_{i,j}\frac{e^{-\lambda u}(\lambda u)^k}{k!}\intd F_i(u) \nonumber \\
        &+\sum_{\substack{k=b-n\\j=1}}^{\substack{\infty\\m}}\int_0^t \Phi_{b-1,j}(t-u)T_{i,j}\frac{e^{-\lambda u}(\lambda u)^k}{k!} \intd F_i(u) \nonumber \\
        &+ \rho_{n,t}(t), \quad, 0< m \leq b,\nonumber
    \end{align}
    where
    \begin{equation*}
        \rho_{n,i}(t)=(1-F_i(t))\begin{cases}
            0, \quad l < k \\
            \frac{e^{-\lambda t}(\lambda t)^{l-n}}{(l-n)!} \quad k \leq l < b \\
            \sum_{k=b-n}^{\infty}\frac{e^{-\lambda t}(\lambda t)^k}{k!}, \quad l=b.
        \end{cases}
    \end{equation*}
    The first term in \eqref{proof11} represents the case when the first service completion time $u$ was before time $t$ and there was no packet loss.
    In this case the number of packets arrivals in $(0,u]$ must be less than $b-n-1$, also the state of the server  might have changed.

    The second term in \eqref{proof11} represents the case when the first service completion time $u$ was before time $t$ and there was a buffer overflow.
    This means that in time $(0,u]$ we had $k\geq b-n$ arrival and $k-b+n$ of them ware dropped, also the state of the server might have changed.

    Finally the third term in \eqref{proof11} corresponds to the situation when the first service completion time is after time $t$.
    This happens with probability $1-F_i(t)$, which depends on the current state of the server.
    In this situation the queue has $l<b$ packets if $n-l$ packets arrived during time $t$ or $b$ packets if $n-b$ or more packets arrived during time $t$.

    Let us now suppose that at the time $t=0$ the queue is empty.
    Applying the total probability formula with respect to the first arrival time we obtain
    \begin{equation}\label{zero}
        \Phi_{0,i}(t)=\int_0^t\Phi_{1,i}(t-u)\lambda e^{-\lambda u}\intd u + \delta_{0,l}e^{-\lambda t}.
    \end{equation}
    The first term in \eqref{zero} corresponds to the situation when the first packet arrives before time $t$.
    The second term in \eqref{zero} corresponds to the situation when there is no arrival before time $t$.
    The probability of such an event is equal to $e^{-\lambda t}$ the queue length in this case stays equal to $0$.

    Applying Laplace transform and matrix notation to the equations \eqref{proof11} and \eqref{zero}, we get:
    \begin{align}\label{proof11}
        &\bm \phi_n(s)=\sum_{k=0}^{b-n-1} \bm \phi_{n+k-1}(s)\bm A_k(s) \\\nonumber
        &+ \sum_{k=b-n}^{\infty} \bm \phi_{b-1}(s)\bm A_k(s) + \bm r_k(s) \\\nonumber
        &\bm \phi_0(s) = \frac{\lambda}{s+\lambda}\bm \phi_1(s) + \frac{\delta_{0,l}}{s+\lambda}
    \end{align}
    Changing indices by substitution $\bm \varphi_n(s)=\bm\phi_{b-n}(s)$ yields:
    \begin{align}
        &\sum_{j=-1}^{n}\bm A_{j+1}(s)\bm\varphi_{n-j}(s)-\bm\varphi_n(s)=\bm\psi_n(s)\label{eq:forpot}\\
        &\bm \varphi_b(s) = \frac{\lambda}{s+\lambda}\bm \varphi_{b-1}(s) + \frac{\delta_{0,l}}{s+\lambda}\label{eq:forpot0},
    \end{align}
    where
    \begin{align}\label{eq:psi}
        &\bm\psi_n(s)=\bm A_{n+1}(s)\bm \varphi_0(s) -\left ( \sum_{k=n+1}^\infty \bm A_k(s)\bm \varphi_1(s)\right ) \\ \nonumber
        &+ \bm r_{b-n}(s).
    \end{align}
    Now, according to lema 3.2.1 from \cite{Chdzinski2007StochMod}, every solution of the system \eqref{eq:forpot} is in the following form:
    \begin{align}\label{eq:solution}
        \bm\varphi_n(s) = \bm R_{n+1}(s)\bm C(s)+\sum_{k=0}^n \bm R_{n-k}(s)\bm\psi_k(s).
    \end{align}
    For $n=0$ in \eqref{eq:solution}, we get
    \begin{equation}\label{eq:c}
        \bm C(s) = \bm A_0(s)\bm \varphi_0(s)
    \end{equation}
    while $n=0$ in \eqref{eq:forpot} yields
    \begin{equation}\label{eq:fi1}
        \bm\varphi_1(s) = \bm{\bar{A}}_0^{-1}(s)(\bm\varphi_0(s) -\bm r_b(s) ).
    \end{equation}
    Note that for the $M/SM/1/b$ system $\bm{\bar{A}}_0(s)=[T_{i,j}f_i(s)]_{i,j}$ which is much simpler in computation then the version for a BMAP queue~\cite{ChydzinskiBMAP2006springer}.

    Using \eqref{eq:c} in \eqref{eq:solution} and \eqref{eq:fi1} in \eqref{eq:psi}, we can represent $\bm\varphi_n(s)$ as a function of $\bm\varphi_0(s)$.
    Now substituting $n=b$ and $n=b-1$ in \eqref{eq:solution} and using \eqref{eq:forpot0} we obtain $\bm\varphi_0(s)$ which finishes the proof of Theorem \ref{th:1}. 
\end{proof}

The Laplace transform of a queue size distribution have the same form as the one for a BMAP queue~\cite{ChydzinskiBMAP2006springer}.
The only difference is in the definition of matrices $\bm A_k(s)$ and vectors $d_k$.
Therefore the numerical complexities of both solutions are the same and equal to $O(m^3b^2)$.
The brute-force solution of \eqref{proof11} has the complexity $O(m^3b^3)$.

The stationary queue size distribution formula is given by:
$$
\Phi(l)=\lim_{t\to\infty}\Phi_{n,i}(t,l)=\lim_{s\to0+}s\phi_{n,i}(s,l),
$$
and does not depend on the initial conditions.
In order to compute the transient state characteristics, one can apply a numerical algorithm e.g~\cite{grassmann2000computational} for the Laplace transform inversion.

Because at the input we have the Poisson process, some formulas get simplified compared to these from~\cite{ChydzinskiBMAP2006springer}.
Let $\bm{\bar{A}}_0(s)$ be an example. Also the functionals $a_{i,k}(s)$ and $d_{i,k}(s)$ are relatively easy to find for some simple service time distributions.

\section{Conclusions and Future Work}\label{sec:cons}
In this paper we presented a novel approach to the finite single server queue with Poisson input and semi-Markov service times.
Using the potential method, we proved the new formula for the laplace transform of queue length distribution in both transient and stationary phase.

Potential method has bean proven to be a powerful tool for computing all kinds of queue characteristic such as waiting time or time to buffer overflow.
Therefore, it should be straightforward to extend our results beyond just a queue length distribution.

Since the Poisson process is not the best traffic model for packet networks.
As a future work we are going to  develop similar formalism for better traffic models such as Markovian Arrival Process.

\section{Acknowledgment}
The work presented in this paper was supported by the national project 647/N-COST/2010/0


\begin{thebibliography}{1}
	
	\bibitem{ChydzinskiBMAP2006springer}
	A.~Chydzinski.
	\newblock Queue size in a bmap queue with finite buffer.
	\newblock In Y.~Koucheryavy, J.~Harju, and V.~Iversen, editors, {\em Next
		Generation Teletraffic and Wired/Wireless Advanced Networking}, volume 4003
	of {\em Lecture Notes in Computer Science}, pages 200--210. Springer Berlin /
	Heidelberg, 2006.
	
	\bibitem{Chdzinski2007StochMod}
	A.~Chydzinski.
	\newblock Time to reach buffer capacity in a bmap queue.
	\newblock {\em Stochastic Models}, 23:195--209, 2007.
	
	\bibitem{springerlink:10.1023/A:1019831410410}
	A.~N. Dudin, V.~I. Klimenok, and G.~V. Tsarenkov.
	\newblock A single-server queueing system with batch markov arrivals,
	semi-markov service, and finite buffer: Its characteristics.
	\newblock {\em Automation and Remote Control}, 63:1285--1297, 2002.
	
	\bibitem{korelacja}
	B.~Emmert, A.~Binzenhöfer, D.~Schlosser, and M.~Weiß.
	\newblock Source traffic characterization for thin client based office
	applications.
	\newblock In A.~Pras and M.~van Sinderen, editors, {\em Dependable and
		Adaptable Networks and Services}, volume 4606 of {\em Lecture Notes in
		Computer Science}, pages 86--94. Springer Berlin / Heidelberg, 2007.
	
	\bibitem{janssen2006applied}
	J.~Janssen and R.~Manca.
	\newblock {\em {Applied semi-Markov processes}}.
	\newblock Springer Science+Business Media, 2006.
	
	\bibitem{grassmann2000computational}
	A.~Joseph, L.~C. Gagan, and W.~Ward.
	\newblock An introduction to numerical transform inversion and its application
	to probability models.
	\newblock In W.~Grassmann, editor, {\em Computational probability},
	International series in operations research \& management science, pages
	257--323. Kluwer Academic, 2000.
	
	\bibitem{neutsSM}
	M.~F. Neuts.
	\newblock The single server queue with poisson input and semi-markov service
	times.
	\newblock {\em J. Appl. Prob}, 1966.
	
\end{thebibliography}

\end{document}
